\documentclass[A4paper]{article}%
\usepackage{graphicx}
\usepackage{amsmath,amssymb,amsfonts}
\usepackage{theorem}
\usepackage{color}
\usepackage{hyperref}
\usepackage[font={small, it}]{caption}
\usepackage[all]{xy}%
\usepackage{amsmath}%
\usepackage{amsfonts}%
\usepackage{amssymb}
\providecommand{\U}[1]{\protect \rule{.1in}{.1in}}
\theoremstyle{change}
\newtheorem{definition}{Definition:}[section]
\newtheorem{proposition}[definition]{Proposition:}
\newtheorem{theorem}[definition]{Theorem:}
\newtheorem{lemma}[definition]{Lemma:}
\newtheorem{corollary}[definition]{Corollary:}
{\theorembodyfont{\rmfamily}
	\newtheorem{remark}[definition]{Remark:}
}
{\theorembodyfont{\rmfamily}
	
}
\newenvironment{proof}
{{\bf Proof:}}
{\qquad \hspace*{\fill} $\Box$}

\newcommand{\fg}{\mathfrak{g}}

\newcommand{\fn}{\mathfrak{n}}

\newcommand{\fh}{\mathfrak{h}}
\newcommand{\fu}{\mathfrak{u}}

\newcommand{\Ad}{\operatorname{Ad}}
\newcommand{\ad}{\operatorname{ad}}

\newcommand{\id}{\operatorname{id}}
\newcommand{\inner}{\operatorname{int}}

\newcommand{\rme}{\mathrm{e}}

\newcommand{\ZC}{\mathcal{Z}}
\newcommand{\UC}{\mathcal{U}}

\newcommand{\XC}{\mathcal{X}}
\newcommand{\DC}{\mathcal{D}}

\newcommand{\C}{\mathbb{C}}

\newcommand{\N}{\mathbb{N}}

\newcommand{\R}{\mathbb{R}}

\begin{document}

\title{The chain control set of a linear control system}
\author{Adriano Da Silva\thanks{Supported by Proyecto UTA Mayor Nº 4768-23} \\
	Departamento de Matem\'atica,\\Universidad de Tarapac\'a - Sede Iquique, Chile.}
\date{\today}
\maketitle

\begin{abstract}
	In this paper, we analyze the chain control sets of linear control systems on connected Lie groups. Our main result shows that the compactness of the central subgroup associated with the drift is a necessary and sufficient condition to assure the uniqueness and compactness of the chain control set.
\end{abstract}

	

\section{Introduction}

The concept of linear control systems appeared, first in \cite{Markus} for matrix groups and subsequently in the works \cite{San, AyTi} for general connected groups, as a natural extension of their counterpart on Euclidean
spaces. Several works address the main issues in control theory for such systems, such as controllability,
observability and optimization appeared over the years, mainly through the results in \cite{JPh1} showing that such generalization is also important for the classification of general control systems on abstract connected manifolds.

In the present paper, we analyze the chain control sets of a linear control system. Understanding chain control sets is vital, since they are intrinsically connected with the dynamics of the system. Chain control sets are closely related to the notion of chain transitive components in the theory of dynamical systems. They contain fixed, periodic, and recurrent points, as well as the maximal regions of controllability, the so-called control sets (see \cite[Chapter 3]{FCWK}). For the class of linear control systems, the dynamics of the drift strongly influence the behavior of the whole system. In fact, several properties concerning controllability and control sets were determined by subgroups induced by the drift (see \cite{DSAy1, DSAyGZ, DS} and references therein). For chain control sets, it is not different; under the assumption that the central subgroup of the drift is compact, the boundedness of such a set only depends on an induced (hyperbolic) system on the nilradical of the group. Since the behavior of the nilpotent part is polynomial, the hyperbolic part can be used to control the growth of the chains, and hence one can show the boundedness of the chain control set. Moreover, compactness also implies that the central subgroup is contained in any chain control set, and hence, there is a unique chain control set in this case.

The paper is structured as follows: In Section 2 we introduce linear vector fields and the decompositions induced by them on the group and the algebra. Linear control systems are also introduced in this section, as are their chain control sets. Some structural results concerning the existence of homomorphisms between groups and conjugations are also proved here. In Section 3 we prove our main results. The first subsection is used to prove the uniqueness of the chain control set. The compactness of the central subgroup implies that one can always construct controlled chains from arbitrary chain control sets to the central subgroup and vice versa. In the second subsection, a detailed analysis of a linear control system on a semi-direct product is done. Although the analysis is quite technical, it is central to obtaining the bounds of the chain control set. The central idea is to look at the noncompact component, which lives on the nilpotent part of the group, and use hyperbolicity in this component to show the boundedness. As a consequence of the previous, in the third and last subsection, we can conclude that any linear control system with a compact central subgroup has a unique, compact chain control set. 

\subsection*{Notations}

Let $G$ be a connected Lie group with Lie algebra $\fg$. For any element $g\in G$ the left and right translations of $G$ by $g$ are denoted by $L_g$ and $R_g$, respectively. The conjugation $C_g$ is the map $C_g=R_{g^{-1}}\circ L_g$. The set of automorphisms of $G$ is denoted by $\mathrm{Aut}(G)$  by $\mathrm{Aut}(\fg)$ we denote the set of automorphisms of $\fg$. The adjoint map $\Ad:G\rightarrow \mathrm{Aut}(\fg)$ is given by the differential $\Ad(g):=(dC_g)_e$, where $e\in G$ stands for the identity element of $G$. For a connected Lie group and a homomorphism $\rho:G\rightarrow\mathrm{Aut}(H)$, the semi-direct product of $G$ and $H$ is the Lie group $G\times_{\rho} H$ whose subjacent manifold is $G\times H$ and the product is given by 
$$(g_1, h_1)(g_2, h_2):=(g_1g_2, h_1\rho(g_1)h_2).$$
Its Lie algebra coincides, as a vector space, to the Cartesian product $\fg\times\fh$. The  {\it Baker-Campbell-Hausdorff (BCH)} formula of two elements $X, Y\in \fg$ is given by
$$\exp(X)\exp(Y)=\exp(c(X, Y)),$$
where $\exp:\fg\rightarrow G$ stands for the exponential map and $c(X, Y)$ is a series depending on $X, Y$ and its brackets. Its first terms are given by
$$c(X, Y)=X+Y+\frac{1}{2}[X, Y]+\frac{1}{12}\left([X, [X, Y]]+[Y, [Y, X]]\right)-\frac{1}{24}[Y, [X, [X, Y]]]+\cdots,$$
and the subsequent terms depend on the brackets of five or more elements. The serie is finite when $\fg$ is a nilpotent Lie algebra. In the nilpotent case, we can endow $\fg$ with the product $(X, Y)\in \fg\times\fg\rightarrow X*Y=c(X, Y)$ such that $(\fg, *)$ is, up to isomorphisms, the simple connected, connected nilpotent Lie group with Lie algebra $\fg$.

\section{Preliminaries}

This section is devoted to providing the background needed to understand the main results presented in the paper. 

\subsection{Decompositions induced by linear vector fields}

Let $G$ a connected Lie group with Lie algebra $\fg$. We say that a vector field $\mathcal{X}$ on $G$ is {\it linear} if its associated flow 
$\{\varphi_t\}_{t\in \R}$ is a $1$-parameter subgroup of $\mathrm{Aut}(G)$. Any linear vector field $\mathcal{X}$ induces a derivation $\mathcal{D}$ of the Lie algebra $\mathfrak{g}$ that satisfy 
\begin{equation*}
(d\varphi_{t})_{e}=\mathrm{e}^{t\mathcal{D}}\; \; \; \mbox{ for all }\; \;
\;t\in \mathbb{R}. \label{derivativeonorigin}
\end{equation*}%
In particular, it holds that
\begin{equation*}
\varphi _{t}(\exp Y)=\exp (\mathrm{e}^{t\mathcal{D}}Y),\mbox{ for all }t\in 
\mathbb{R},Y\in \mathfrak{g}.
\end{equation*}

Let $\alpha\in\C$ an eigenvalue of $\DC$ and consider the real generalized eigenspaces of $\DC$ given by  
$$
\mathfrak{g}_{\alpha}=\{X\in \mathfrak{g}:(\mathcal{D}-\alpha I)^{n}X=0\;\;\mbox{for some }n\geq 1\}, \;\;\mbox{ if }\;\;\alpha\in\R \;\;\mbox{ and}
$$
$$\mathfrak{g}_{\alpha}=\mathrm{span}\{\mathrm{Re}(v), \mathrm{Im}(v);\;\;v\in \bar{\fg}_{\alpha}\},\;\;\mbox{ if }\;\;\alpha\in\C$$
where $\bar{\fg}=\fg+i\fg$ is the complexification of $\fg$ and $\bar{\fg}_{\alpha}$ the generalized eigenspace of the extension of $\DC$ to $\bar{\fg}$ given by $\bar{\DC}=\DC+i\DC$. 

Since $[\bar{\fg}_{\alpha },\bar{\fg}_{\beta }]\subset 
\bar{\fg}_{\alpha +\beta }$ when $\alpha +\beta $ is an eigenvalue of $\mathcal{D}$ and zero otherwise (see \cite[Proposition 3.1]{SM1}), we get for the real case that
$$[\fg_{\lambda_1}, \fg_{\lambda_1}]\subset \fg_{\lambda_1+\lambda_2}\;\;\;\mbox{ when }\lambda_1+\lambda_2=\mathrm{Re}(\alpha)\;\mbox{ for some eigenvalue }\alpha\;\mbox{ of }\;\DC\;\mbox{ and zero otherwise},$$
where $\fg_{\lambda}:=\bigoplus_{\alpha; \mathrm{Re}(\alpha)=\lambda}\fg_{\alpha}$, with $\fg_{\lambda}=\{0\}$ if $\lambda\in\R$ is not the real part of any eigenvalue of $\DC$. 

The {\it unstable, central, }and {\it stable} subalgebras of $\fg$ are given, respectively, by
\begin{equation*}
\mathfrak{g}^{+}=\bigoplus_{\alpha :\, \mathrm{Re}(\alpha)>	0}\mathfrak{g}_{\alpha },\hspace{1cm}\mathfrak{g}^{0}=\bigoplus_{\alpha :\,%
	\mathrm{Re}(\alpha )=0}\mathfrak{g}_{\alpha }\hspace{1cm}%
\mbox{ and }\hspace{1cm}\mathfrak{g}^{-}=\bigoplus_{\alpha :\, \mathrm{Re}%
	(\alpha )<0}\mathfrak{g}_{\alpha }.
\end{equation*}

It is straightforward to see that  $\mathfrak{g}=\mathfrak{g}^{+}\oplus \mathfrak{g}^{0}\oplus \mathfrak{g}^{-}$. Moreover, $\mathfrak{g}^{+},\mathfrak{g}^{0},\mathfrak{g}^{-}$ are $\DC$-invariant Lie subalgebras with $\mathfrak{g}^{+}$, $\mathfrak{g}^{-}$ nilpotent ones. The derivation is called {\it hyperbolic} if $\fg=\fg^{+, -}:=\fg^+\oplus\fg^-$, that is, if $\DC$ has only eigenvalues with nonzero real parts.

Integration of the previous subalgebras to $G$ gives rise to the connected subgroups $G^{+}$, $G^{-}$, $G^{0}$, $G^{+,0},$ and $G^{-,0}$ whose Lie algebras are given by $\mathfrak{g}^{+}$, $%
\mathfrak{g}^{-}$, $\mathfrak{g}^{0}$, $\mathfrak{g}^{+,0}:=\mathfrak{g}%
^{+}\oplus \mathfrak{g}^{0}$ and $\mathfrak{g}^{-,0}:=\mathfrak{g}^{-}\oplus 
\mathfrak{g}^{0}$, respectively. By Proposition 2.9 of \cite{DS}, all the above subgroups are $\varphi_t$-invariant, closed, and their intersection are trivial, that is, 
$$G^+\cap G^-=G^+\cap G^{-, 0}=\ldots=\{e\}.$$
Moreover, $G^+$ and $G^-$ are simply connected, connected nilpotent Lie groups. We call the subgroups $G^+, G^0$ and $G^-$ the {\it unstable, central,} and {\it stable subgroups} of $\XC$, respectively. We also denote by $G^{+, -}$ the product of $G^+$ and $G^-$, that is, $G^{+, -}=G^+G^-$.

The group $G$ is said to be {\it decomposable} if,  
$$G=G^{+, 0}G^-=G^{-, 0}G^+=G^{+, -}G^0.$$ 
If $G^0$ is a compact subgroup, then $G$ is decomposable (see \cite[Proposition 3.3]{DSAyGZ}). 

\begin{remark}
	Let us note that, on decomposable groups, any element can be written, uniquely, as a product of elements in the stable, central and unstable subgroups.
\end{remark}

The next result shows that when $G^0$ is a compact Lie group, the linear vector field $\XC$ can be projected onto a semidirect product. It will be important to reduce our general case to a particular one.

\begin{lemma}
	\label{conj}
Let $\XC$ be a linear vector field with a compact central subgroup $G^0$ on $G$. Then, there exists a compact subgroup $H$, acting by automorphisms on a simply connected, connected nilpotent Lie group $\fu$ and a homomorphism 
 $$\psi:G\rightarrow H\times_{\rho}\fu,$$ 
 that conjugates $\XC$ to a linear vector field on $H\times_{\rho}\fu$ of the form $\widehat{\XC}\times\widehat{\DC}$ such that $\widehat{\DC}$ is hyperbolic.
\end{lemma}

\begin{proof} Let us denote by $\fn$ the nilradical of $\fg$ and $N$ its associated connected group on $G$. Since $G^0$ is a compact subgroup, the intersection $G^0\cap N=N_0$ is a compact, connected subgroup of the center $Z(N)$ of $N$ whose Lie algebra satisfies $\fn=\fg^{+, -}\oplus\fn^0$ (see \cite[Lemma 2.1]{DSAy1}). In particular, $N_0$ is an ideal of $G$ and we can define 
$$H=G^0/N_0\hspace{.5cm}\mbox{ and }\hspace{.5cm} \fu=\fn/\fn_0,$$
where $\fn_0\subset\fn$ is the Lie algebra of $N_0$. The group $H$ is a compact Lie group and $\fu$ is a connected, simply connected nilpotent Lie group when we endow it with the product induced by the BCH formula.

Define 
$$\rho:H\rightarrow\mathrm{Aut}(\fu), \hspace{1cm}\rho(hN_0)(X+\fn_0):=\Ad(h)X+\fn_0.$$ 
It is straightforward to see that $\rho$ is well defined. Moreover, $\rho$ is continuous (and hence differentiable) since it satisfies $\rho\circ\pi_1=\pi_2\circ\Ad$, where $\pi_1:G\rightarrow G/N_0$ and $\pi_2: \fn\rightarrow\fn_0$ are the canonical projections. Therefore, the semi-direct product $H\times_{\rho}\fu$ is well defined. 

On the other hand, the compactness of $G^0$ implies that $G$ is decomposable. Therefore, any $g\in G$ is uniquely decomposed as $g=kh$, with $h\in G^0$ and $k\in G^{+, -}$.  Since $\fn$ is nilpotent, there exists  $X\in \fn$ such that $k=\rme^X$. On the other hand, the fact that $\exp:\fn\rightarrow N$ is a covering map implies that $\exp^{-1}(e)\subset\fn$ is a discrete subset. Hence
$$\rme^{t\DC}(\exp^{-1}(e))\subset \exp^{-1}(e), \;\;\;\mbox{ for all }\;\;\;t\in\R,$$
forces that $\exp^{-1}(e)\subset\ker\DC\cap\fn=\fn_0$, showing that 
$$\rme^{X_1}=\rme^{X_2}\iff -X_1*X_2\in\fn_0\iff X_1+\fn_0=X_2+\fn_0,$$
where the last equivalence is by considering the product on $\fn/\fn_0$ induced by the BCH-formula.

Define the map 
	$$\psi: g\in G\mapsto (hN_0, X+\fn_0)\in H\times_{\rho}\fu.$$
According to the previous comments, the map $\psi$ is well defined. Moreover, since the previous decompositions are continuous, $\psi$ is also continuous.  

Moreover, if $g_1=\rme^{X_1}h_1$ and $g_2=\rme^{X_2}h_2$, then
$$\rme^{X_1}h_1\rme^{X_2}h_2=\rme^{X_1}h_1\rme^{X_2}h_1^{-1}(h_1h_2)=\rme^{X_1*\Ad(h_1)X_2}h_1h_2,$$
implying that
$$\psi(g_1g_2)=\psi\left(\rme^{X_1*\Ad(h_1)X_2}h_1h_2\right)=(h_1h_2N_0, X_1*\Ad(h_1)X_2+\fn_0)$$
$$(h_1N_0h_2N_0, (X_1+\fn_0)*(\Ad(h_1)X_2+\fn_0))=(h_1N_0h_2N_0, (X_1+\fn_0)*\rho(h_1N_0)(X_2+\fn_0))$$
$$=(h_1N_0, X_1+\fn_0)(h_2N_0, X_2+\fn_0)=\psi(g_1)\psi(g_2),$$
showing that $\psi$ is a homomorphism.

To conclude, let us show the assertion concerning the induced vector field. By the previous decomposition, we have that 
$$\varphi_t(g)=\varphi_t(\rme^{X}h)=\varphi_t(\rme^{X})\varphi_t(h)=\exp(\rme^{t\DC}X)\varphi_t(h),$$
and hence, 
$$\psi(\varphi_t(g))=\left(\varphi_t(h)N_0, \rme^{t\DC}X+\fn_0\right)=(\pi_1(\varphi_t(h))), \pi_2( \rme^{t\DC|_{\fn}}X)),$$
showing that the flow on $H\times_{\rho}\fu$ induced by $\psi$ is given by 
$$(\pi_1\circ\varphi_t|_{G^0})\times (\pi_2\circ \rme^{t\DC|_{\fn}})\hspace{.5cm}\implies\hspace{.5cm}.$$
By derivation, we obtain that $\XC$ is $\psi$-conjugated to the linear vector field $\widehat{\XC}\times\widehat{\DC}$ on $H\times_{\rho}\fu$ determined by the relations 
$$\widehat{\XC}\circ\pi_1=(\pi_1)_*\circ\XC|_{G^0}\hspace{.5cm}\mbox{ and }\hspace{.5cm} \widehat{\DC}\circ\pi_2=\pi_2\circ\DC|_{\fn},$$
where the second equality follows from the fact that $\pi_2$ is a linear map. Moroever, since $\fn=\fg^{+, -}\oplus \fn^0$ we get by invariance that $\pi_2\circ\DC|_{\fn}=\pi_2\circ\DC|_{\fg^{+, -}}$, which implies that $\widehat{\DC}$ is hyperbolic and concludes the proof.
\end{proof}

\begin{remark}
    Note that the eigenvalues of $\widehat{\DC}$ are precisely the eigenvalues with a nonzero real part of $\DC$.
\end{remark}

\subsection{Linear control systems}

Let $G$ be a connected Lie group with Lie algebra $\fg$ identified with the right-invariant vector fields and $\Omega\subset\R^m$ a compact and convex subset such that $0\in\inner\Omega$. We define the {\it set of control functions} by 
$$\UC:=\{u:\R\rightarrow\R^m; \;u\;\mbox{ measurable and  }  u(t)\subset\Omega \mbox{ for a.a. }t\in \R\}.$$
A \emph{linear control system (LCS)} on $G$ is given a family of ordinary differential equations 
\begin{flalign*}
&&\dot{g}(t)=\XC(g(t))+\sum_{j=1}^mu_j(t)Y^j(g(t)),  &&\hspace{-1cm}\left(\Sigma_G\right)
\end{flalign*}
where $\XC$ is a linear vector field, $Y^1, \ldots, Y^m\in\fg$ are right-invariant vector fields and $u=(u_1, \ldots, u_m)\in\UC$. For any $g\in G$ and $u\in\UC$, the solution $t\mapsto\phi(t, g, u)$ of $\Sigma_G$ is complete and satisfies the cocycle property
$$\phi(t+s,, g, u)=\phi(t, \phi(s, g, u), \theta_su),$$
where $\theta_su:=u(\cdot+s)$ is the shift flow. By our hypothesis, the set $\UC$ is a compact metric space, the shift $\theta$ is a continuous cocycle and the map $u\in\UC\mapsto \phi(t, x, u)$ is continuous. Moreover, it holds that  
\begin{equation}
\label{prop}
\phi(t, hg, u)=L_{\phi(t, h, u)}\left(\varphi_t(g)\right)=\phi(t, h, u)\varphi_t(g), \;\;\mbox{ for any }\;\;g, h\in G.
\end{equation}

Linear control systems appear as generalizations of their counterparts for Euclidean spaces. Moreover, as shown in \cite{JPh1}, this class of system is a model for any control-affine on a general connected manifold that is transitive and generates a finite-dimensional Lie algebra. 

In order to understand the dynamics of LCSs, understanding chain control sets is vital. Chain control sets are closely related to the notion of chain transitive components in the theory of dynamical systems. They contain fixed, periodic, and recurrent points of the system and also the maximal regions of controllability, the so called control sets (for the main properties of chain control sets, the reader should consult \cite[Chapter 3 and 4]{FCWK}). 

Let $d$ be a left-invariant metric on $G$. For any $x, y\in G$ and $\varepsilon, \tau>0$, a {\it controlled $(\varepsilon, \tau)$-chain} from $x$ to $y\in G$ is given by $n\in\N$, $x_0, x_1 \ldots, x_n$ in $G$, $u_0, \ldots, u_{n-1}\in \UC$ and $T_0, \ldots, T_{n-1}\geq \tau$ with $x_0=x$ and $x_n=y$ and 
$$d(\phi(\tau_i, x_i, u_i), x_{i+1})<\varepsilon, \hspace{.5cm}\mbox{ for all }\hspace{.5cm}i=0, 1, \ldots, n-1.$$

The points $x_1, \ldots, x_{n-1}$ are called the {\it jumps} of the controlled chain, and the curves 
$$s\in [0, T_i]\mapsto\phi(s, x_i, u_i), \hspace{1cm}i=0, \ldots, n-1,$$
the {\it pieces of trajectories} of the chain. 

\begin{remark}
\label{trivialjumps}
   By adding ``trivial" jumps on a controlled $(\varepsilon, \tau)$-chain, one can assume w.l.o.g. that the times of a controlled $(\varepsilon, \tau)$-chain belong to $[\tau, 2\tau]$. 
\end{remark}

\begin{definition}
\label{chaincontrol}
	A  set $E\subset G$ is called a {\it chain control set} of the system $\Sigma_G$ if it satisfies:
 \begin{itemize}
     \item[(i)] for all $x\in E$, there exists $u\in\UC$ such that $\phi(\R, x, u)\subset E$;
     \item[(ii)] for all $x, y\in E$ and $\varepsilon, \tau>0$ there is a controlled $(\varepsilon, \tau)$-chain from $x$ to $y$;
     \item[(iii)] $E$ is maximal, with relation to set inclusion, with properties (i) and (ii). 
 \end{itemize}
\end{definition}

\begin{remark}
It is not hard to see that the previous definition is independent of the chosen left-invariant metric. In fact, the definition is the same if we choose instead a right-invariant metric. Moreover, as shown in \cite[Corollary 4.3.12]{FCWK}, chain control sets are closed.
\end{remark}

Let us denote by $E_{\varepsilon, \tau}$ the set of jumps of chains connecting points in the chain control set $E$. Since any two points in $E$ can be connected by controlled $(\varepsilon, \tau)$-chains, it holds that 
$$E\subset \bigcap_{\varepsilon, \tau}E_{\varepsilon, \tau}.$$
Moreover, the set 
$$N_{\phi}(E_{\varepsilon, \tau}):=\phi\left([0, 2\tau], E_{\varepsilon, \tau}, \UC\right).$$
contains all the pieces of trajectories of controlled $(\varepsilon, \tau)$-chains connecting points in $E$ (see Remark \ref{trivialjumps}).

\begin{remark}
	If $F$ is $\varphi$-invariant, then for any $\varepsilon, \tau>0$ we have that 
      $$x_0=\varphi_{-\tau}(g), x_1=g \hspace{.5cm}\mbox{ and }\hspace{.5cm}y_0=g, y_1=\varphi_{\tau}(g),$$
      are $(\varepsilon, \tau)$-chain from $\varphi_{-\tau}(g)\in F$ to $g$ and from $g$ to $\varphi_{\tau}(g)\in F$.
\end{remark}

Let us consider connected Lie groups $G, H$ and a surjective homomorphism $\psi:G\rightarrow H$. We say that two linear control systems $\Sigma_G$ and $\Sigma_H$ on $G$ and $H$, respectively, are $\psi$-conjugated if 
\begin{equation}
\label{solu}
\psi\left(\phi^G(t, g, u)\right)=\phi^H(t, \psi(g), u)), \;\;\;\mbox{ for any }\;\;g\in G, t\in\R\;\mbox{ and }u\in\UC.
\end{equation}

The next result analyzes the relationship between chain control sets and their image for conjugated systems.

\begin{proposition}
	\label{control}
Let $\psi$ be a conjugation between the linear control system $\Sigma_G$ and $\Sigma_H$ as previously. If $E$ is a chain control set of $\Sigma_G$, there exists a chain control set $F$ of $\Sigma_H$ such that $\psi(E)\subset F$. In particular, if $F$ and $\ker\psi$ are compact subsets, then $E$ is compact.
\end{proposition}

\begin{proof}
Let us consider $\varepsilon, \tau>0$. By continuity, there exists $\delta>0$ such that $\psi(B_{\delta}(e_G))\subset B_{\varepsilon}(e_G)$, where the metrics $d_G$ on $G$ and $d_H$ on $H$ are left-invariant. 
Let $x, y\in G$ and consider $n\in\N$, $\tau_0, \ldots, \tau_{n-1}\geq\tau$, $x=x_0, \ldots, x_n=y$, $u_0, \ldots, u_{n-1}\in\UC$ be a controlled $(\delta,\tau)$-chain from $x$ to $y$. By definition, for $i=0, \ldots, n-1$, 
$$d_G(\phi^G(\tau_i, x_i, u_i), x_{i+1})<\delta\hspace{.5cm}\iff\hspace{.5cm}\phi^G(\tau_i, x_i, u_i)\in B_{\delta}(x_{i+1})(e_G).$$
and by the invariance of the metrics,
$$\psi\left(B_{\delta}(x_{i+1})\right)=\psi\left(x_{i+1}B_{\delta}(e)\right)=\psi(x_{i+1})\psi\left(B_{\delta}(e)\right)\subset \psi(x_{i+1})B_{\varepsilon}(e)=B_{\varepsilon}(\psi(x_{i+1})),$$
implying that,
$$\phi^H(\tau_i, \psi(x_i), u_i)\stackrel{(\ref{solu})}{=}\psi(\phi^G(\tau_i, x_i, u_i))\in \psi\left(B_{\delta}(x_{i+1})\right)\subset B_{\varepsilon}(\psi(x_{i+1})).$$
Therefore, $\psi(E)$ satisfies condition (ii) in the Definition \ref{chaincontrol}. Since condition (i) is trivially satisfied on $\psi(E)$ by the conjugation property, we have that $\psi(E)$ has to be contained in a chain control set $F$ of $\Sigma_H$, showing the assertion.

On the other hand, if $\ker\psi$ is a compact subgroup, then $\psi$ is a proper map, and since 
$$E\subset \psi^{-1}\left(\psi\left(E\right)\right)\subset \psi^{-1}\left(F\right),$$
our assertion on compactness follows.
\end{proof}

\section{The main results}

Our aim in this section is to prove our main results, showing that the compactness of the central subgroup implies the uniqueness and compactness of the chain control set containing the identity.

\subsection{Uniqueness of chain control sets}

In this section, we show that if $G^0$ is compact subgroup, then $\Sigma_G$ admits a unique chain control set containing $G^0$. 

Let $H$ be a compact, connected Lie group and $\XC$ a linear vector field on $H$. Following \cite[Chapter 11]{SMGr}), if $Z(H)$ is the center of $H$ and $H'$ the derivated group, then $Z(H)$ is a connected, compact abelian Lie group, and $H'$ is a connected, compact semisimple Lie group satisfying
$H=Z(H)H'$.

Since both $Z(H)$ and $H'$ are invariant by automorphisms, the flow $\{\varphi_t\}_{t\in\R}$ restricts to a flow of automorphisms on $Z(H)$ and $H'$. As a consequence, the fact that the group automorphisms of a compact abelian Lie group is discrete implies that $\varphi_t|_{Z(H)}=\id_{Z(H)}$. 

On the other hand, the fact that any derivation on a semisimple Lie group is inner, implies that $\DC|_{\fh'}$ is inner, where $\fh'$ is the Lie algebra of $H'$. Therefore, $\DC|_{\fh'}$ is skew-symmetric for the Cartan-Killing form, and hence, for some left-invariant metric, $\varphi_t|_{H'}$ acts as a flow of isometries \cite[Theorem 2.4]{DSAyPH}. As a conclusion, there exists a left-invariant metric $d$ on $H$ such that $\varphi_t$ is a flow of isometries. In particular, for any $h\in H$ and any $\varepsilon, \tau>0$ there exists $T\geq \tau$ such that $d(\varphi_T(h), h)<\varepsilon$.

Using the previous, we can prove our main result of this section.

\begin{theorem}
\label{unique}
     If $G^0$ is a compact subgroup, then $\Sigma_G$ admits only one chain control set $E$ satisfying $G^0\subset E$.
\end{theorem}
   
\begin{proof} Let us start by proving that $G^0$ is contained in a chain control set.

Since $G^0$ is invariant by the flow of $\XC$, it certainly satisfies condition (i) in Definition \ref{chaincontrol} for the trivial control. Therefore, we only have to show that any two points on $G^0$ can be connected by a controlled $(\varepsilon, \tau)$-chain.

Let $d$ be a left-invariant metric on $G$. Since $G^0$ is compact, any two metrics on $G^0$ are equivalent. Therefore, by the previous discussion, for any $g\in G^0$ and any $\varepsilon, \tau>0$, there exists $T\geq\tau$ such that $d(\varphi_T(g), g)<\varepsilon$. 

Let $x, y\in G^0$. Since $G^0$ is path connected, there exists a curve $\gamma$ connecting $x$ to $y$. Cover the compact $\mathrm{Im}(\gamma)$ with a finite number of $\varepsilon/4$-balls, with centers $x_0=x, x_1, \ldots, x_n=y$ satisfying $d(x_i, x_{i+1})<\varepsilon/2$. Consider $T_i\geq\tau$ the times satisfying $d(\varphi_{T_i}(x_i), x_i)<\varepsilon/2$. For $u_i\equiv 0$, we get that
$$d(\phi(T_i, x_i, u_i), x_{i+1})=d(\varphi_{T_i}(x_i), x_{i+1})\leq d(\varphi_{T_i}(x_i), x_i)+d(x_i, x_{i+1})<\varepsilon,$$
showing that $G^0$ satisfies condition (ii) of Definition \ref{chaincontrol}. As a consequence, there exists a chain control set $E$ such that $G^0\subset E$.
\bigskip

Let $\widehat{E}$ be also a chain control set for $\Sigma_G$. The union $E\cup\widehat{E}$ certainly satisfies condition (i) in the Definition \ref{chaincontrol}. Therefore, if we show that it also satisfies condition (ii) of the same definition, maximality implies that $E=\widehat{E}$. In order to show the former, it is enough to construct, for any $\varepsilon, \tau>0$, controlled $(\varepsilon, \tau)$-chains from $G^0$ to $\widehat{E}$ and from $\widehat{E}$ to $G^0$.

Let then be $\varepsilon, \tau>0$ and $x\in\widehat{E}$. By condition (i) in Definition \ref{chaincontrol} there exists $u\in\UC$ such that $\phi(\R, x, u)\subset \widehat{E}$. 

By the compactness of $G^0$, we can decompose (uniquely) $x=gkh$ with $g\in G^+$, $k\in G^0$ and $h\in G^-$. Using the left-invariance of the metric, we get that 
$$d(\phi(t, x, u), \phi(t, gk, u))\stackrel{(\ref{prop})}{=}d(\phi(t, gk, u)\varphi_t(h), \phi(t, gk, u))=d(\varphi_t(h), e)\rightarrow 0, \hspace{.5cm}t\rightarrow +\infty,$$
and hence, 
$$\exists t^*\geq\tau, \hspace{.5cm}d(\phi(t^*, x, u), \phi(t^*, gk, u))<\varepsilon.$$ 

On the other hand, the compactness of $G^0$ and the fact that $\varphi_{-t}(g)\rightarrow 0$ as $t\rightarrow+\infty$, implies that 
$$\exists t^{**}\geq\tau, k'\in G^0 \hspace{.5cm}d(\varphi_{-t^{**}}(gk), k')<\varepsilon.$$ 

Therefore, the points
$x_0=\varphi_{-\tau}(k')\in G^0, x_1=\varphi_{-t^{**}}(gk), x_2=gk, x_3=\phi(t^*, x, u)$, the control functions $u_0=u_1\equiv 0, u_2=u$ and the times $T_0=\tau, T_1=t^{**}, T_2=t^*$ form a controlled $(\varepsilon, \tau)$-chain from a point in $G^0$ to a point in $\widehat{E}$. 

Analogously, let us consider the decomposition $x=abc$, with $a\in G^-$, $b\in G^0$ and $c\in G^+$. Since
$$d(\phi(-t, x, u), \phi(-t, ab, u))\stackrel{(\ref{prop})}{=}d(\phi(-t, ab, u)\varphi_{-t}(c), \phi(-t, ab, u))=d(\varphi_{-t}(c), e)\rightarrow 0, \hspace{.5cm}t\rightarrow +\infty,$$
we get that
$$\exists t'\geq\tau, \hspace{.5cm}d(\phi(-t', x, u), \phi(-t', ab, u))<\varepsilon.$$ 
Using that $\varphi_{t}(a)\rightarrow 0$ as $t\rightarrow+\infty$ and that $G^0$ is a compact subgroup, assures that 
$$\exists t^{''}\geq\tau, b'\in G^0 \hspace{.5cm}d(\varphi_{-t^{''}}(ab), b')<\varepsilon.$$ 
Therefore, by the cocycle property, the points,
$x_0=\phi(-(t'+\tau), x, u)\in \widehat{E}, x_1=\phi(-t', ab, u), x_2=ab, x_3=b'$, the control functions $u_0=\theta_{-(t'+\tau)}u, u_1=\theta_{-t'}u, u_2\equiv 0$ and the times $T_0=\tau, T_1=t', T_2=t^{''}$ form a controlled $(\varepsilon, \tau)$-chain from a point in $\widehat{E}$ to a point in $G^0$, concluding the proof.
\end{proof}

\subsection{LCSs on a semi-direct product}

Let $\fu$ be a nilpotent Lie algebra and consider the connected, simply connected Lie group $(\fu, *)$, where
$$X*Y=c(X, Y), \;\; X, Y\in\fu.$$
This identification between the subjacent manifold of $(\fu, *)$ with $\fu$, allows us to work indistinctly with their elements. However, in order to make a distinction between them, we will use capital letters $X, Y, Z, \ldots$ for the elements in  the Lie algebra $(\fu, [\cdot, \cdot])$ and small letters $x, y, z, \ldots$ for the elements of the Lie group $(\fu, [\cdot, \cdot])$. Moreover, we will denote both, group and algebra, only by $\fu$.

Under the previous identification, the group of automorphisms of $\fu$ as a Lie group is given by $\mathrm{Aut}(\fu)$ and a linear vector field on $\fu$ coincides with its associated derivation. 

Let $H$ be a connected Lie group and $\rho:H\rightarrow\mathrm{Aut}(\fu)$ a continuous map, and define the semi-direct product $H\times_{\rho}\fu$. An arbitrary  right-invariant vector field on $H\times_{\rho}\fu$ has the expression
$$
(Y, Z)(h, x)=\left(Y(g), (\rho(h)Z)(x)\right), \;\mbox{ where }\; (\rho(h)Z)(x):=\sum_{p=0}^{k-1}c_p\ad(x)^p\rho(h)Z,$$
and the coefficients  $c_j$ are the ones given by the BCH formula. In particular, if $\XC$ is a linear vector field on $H$ and $\DC$ a derivation of $\fu$, then $\XC\times\DC$ is a linear vector field on $H\times_{\rho}\fu$ and we can consider a LCS on $H\times_{\rho}\fu$ is given by 
\begin{flalign*}
\label{semi}
&&\, \left\{
\begin{array}{l}
\dot{h}=\XC(h)+\sum_{j=1}^mu_jY_j(h)\\
\dot{x}=\DC x+\sum_{j=1}^mu_j(\rho(h)Z_j)(x).
\end{array}
\right. &&\hspace{-1cm}\left(\Sigma_{H\times_{\rho}\fu}\right).
\end{flalign*}

Our aim in this section is to study the central chain recurrent set of $\Sigma_{H\times_{\rho}\fu}$. In order to do that, we decompose $\fu$ as follows:

Let us consider the central series $\fu:=\fu^1\supset \fu^2\supset\cdots\supset\fu^k\supset\fu^{k+1}=\{0\}$, where 
$$\fu^{i+1}=[\fu^i, \fu], \hspace{.5cm}\mbox{ for }\hspace{.5cm}i\in\{1, \ldots, k\},$$
and for any $i\in\{1, \ldots, k\}$ let $V_i\subset \fu^i$ to be a vector subspace that complements $\fu^{i+1}$ in $\fu^i$, that is, 
$$
V_i\oplus\fu^{i+1}=\fu^i.\hspace{1cm}\mbox{ In particular, }\;\fu=\bigoplus_{l=i}^kV_l, \;\;\mbox{ for }\;\;i\in\{1, \ldots, k\}.
$$
For any $x\in\fu$, we write $x=(x^1, \ldots, x^k)$ to emphasize the decomposition into the components $V_i$-components of $x$. Any linear map $T:\fu\rightarrow\fu$ satisfying $T\fu^i\subset\fu^i$, for $i=1, \ldots, k$  have a nice block-decomposition onto the $V_i$-components. In fact, one can write  
\begin{equation}
\label{block}
T=\left(\begin{array}{ccccc}
T_{11} &     0    &     0    & \cdots &    0    \\
T_{21} & T_{22} &     0    & \cdots &    0    \\
T_{31} & T_{32} & T_{33} & \cdots &    0    \\
\vdots   & \vdots   & \vdots   & \ddots &  \vdots \\
T_{k1} & T_{k2} & T_{k3} & \cdots & T_{kk}
\end{array}\right), \;\;\mbox{ where }\;\;T_{ij}:V_j\rightarrow V_i\;\;\mbox{ is a linear map}.
\end{equation}

Derivations and automorphisms of $\fu$ have the block decomposition in (\ref{block}). For the particular case of inner derivations $\ad(x)$ for $x\in\fu$, the previous block decomposition have zeros on the diagonal. 

Using the block decomposition of the derivation $\DC$ associated with the second equation, the authors showed in \cite[Theorem 3.2]{DSAy1} that the solution of the second equation of $\Sigma_{H\times_{\rho}\fu}$ associated with a control function $u\in\UC$ and starting at a point $(h, x)\in H\times_{\rho}\fh$, with $x=(x^1, \ldots, x^k)$, are given component-wise as

	\begin{equation}
	\label{solutiondirect}
	x_{t, u, h}^i=\int_0^t\rme^{(t-s)\DC_{ii}}G^i\left(x^1_{s, u, h}, \ldots, x^{i-1}_{s, u, h}; Z_{s, u, h}\right)ds+\rme^{t\DC_{ii}}x^i, \;\;\mbox{ for }\;\;i=1, \ldots, k,
	\end{equation}

where $G^i:V_1\times\cdots\times V_{i-1}\times \fu\rightarrow V_i$ are continuous maps and $Z_{t, u, h}:=\rho(h_{t, u})\left(\sum_{j=1}^mu_j(t)Z_j\right)$, with $s\mapsto h_{s, u}$ is the first component of the solution of $\Sigma_{H\times_{\rho}\fu}$ associated with the control $u\in\UC$ and starting at $h$.

The next result analyze the set projection of the set containing all the pieces of trajectories of controlled $(\varepsilon, \tau)$-chains.

\begin{proposition}
\label{auxiliar}
For any  $i\in\{1, \ldots, k\}$, we have that 
$$\prod_{p=1}^{i-1} \pi_p(E_{\varepsilon, \tau})\hspace{.5cm}\mbox{ bounded }\hspace{.5cm}\implies\hspace{.5cm}\prod_{p=1}^{i-1} \pi_p\left(N_{\phi}(E_{\varepsilon, \tau})\right)\hspace{.5cm}\mbox{ bounded.}$$

\end{proposition}

\begin{proof}  Since, 
 $$x^p_{t, u, h}=\int_0^t\rme^{(t-s)\DC_{pp}}G^p\left(x^1_{s, u, h}, \ldots, x^{p-1}_{s, u, h}; Z_{s, u, h}\right)ds+\rme^{t\DC_{pp}}x^p,$$
we get, 
 $$\pi_p\left(N_{\phi}(E_{\varepsilon, \tau})\right)=\{x^{p}_{s, u, h}, \hspace{.2cm}u\in\UC, \hspace{.2cm}s\in [0, 2\tau], \hspace{.2cm}h\in H\hspace{.2cm}\mbox{ with }\hspace{.2cm}x^{p}_{0, u, h}\in \pi_p(E_{\varepsilon, \tau})\},$$
implying that $\overline{\pi_p\left(N_{\phi}(E_{\varepsilon, \tau})\right)}$ is a compact set as soon as
 $$G^p:\pi_1\left(N_{\phi}(E_{\varepsilon, \tau})\right)\times\cdots\times \pi_{p-1}\left(N_{\phi}(E_{\varepsilon, \tau})\right)\times \ZC\rightarrow V_p,$$
 is a bounded map, where $\ZC:=\{Z_{h, u}, h\in H, u\in\UC\}$. However, for $p=1$ it holds that  $G^1(\ZC)$ is a compact set, by the continuity of $G^1$, and the compactness of $H\times\UC$, implying the claim.

\end{proof}

In what follows, we use the block decomposition to show that the product in $\fu$ also has a nice decomposition into the $V_i$-components. Such knowledge will be necessary in order to prove the boundedness of the chain control set of $\Sigma_{H\times_{\rho}\fu}$. In order to do that, we prove the next lemma, which is a generalization of \cite[Lemma 3.1]{DSAy1}.

\begin{lemma}
\label{derivations}
For any $p\in\{1, \ldots, k-1\}$ and $x_1, \ldots, x_p\in\fu$, the maps $B^p_{ij}(x_1, \ldots, x_p):V_j\rightarrow V_i$ associated with the block decomposition of the derivation $\ad(x_1)\circ\cdots\circ\ad(x_p)$, satisfies 
	$$B^p_{ij}(x_1, \ldots, x_p)=\left\{\begin{array}{cc} 0 & \;\mbox{ for }\;i<p+j\\B_{ij}^p\left(x_1^1, \ldots, x_1^{i-j-p+1}, \ldots, x_p^1, \ldots, x_p^{i-j-p+1}\right) & \;\mbox{ for }\;i\geq p+j \end{array}\right.,$$
	when $x_i=(x_i^1, \ldots, x_i^k)$ for $i=1, \ldots, p$.
\end{lemma}
\begin{proof}
	Let us prove the lemma by induction on $p\in\{1, \ldots, k-1\}$. Since $V_i\subset \fu^i$, it holds that
 $$\forall x^l\in V_l, \hspace{1cm}\ad(x^l) V_j\subset\fu^{j+l}=\bigoplus_{q={j+l}}^kV_q.$$ 
 Therefore, $B_{ij}(x^l)=0$ for any $x^l\in V_l$ if $i<l+j$ and hence,
$$\ad(x)=\sum_{l=1}^{k-1}\ad(x^l)\;\;\; \implies \;\;\;B_{ij}(x)=\left\{\begin{array}{cc} 0 & \;\mbox{ for }\;i<j+1\\ B_{ij}\left(x^1, \cdots, x^{i-j}\right) & \;\mbox{ for }\;i\geq j+1 \end{array}\right.,$$
showing the result for $p=1$. Let us assume that the lemma holds for all $i\leq p$, with $p\in\{1, \ldots k-1\}$. Since, 
$$B^{p+1}_{ij}(x_1, \ldots, x_p, x_{p+1})=\sum_{l=1}^kB^p_{il}(x_1, \ldots, x_p)B_{lj}( x_{p+1}),$$
and by the inductive hypothesis,
$$B_{lj}(x_{p+1})=\left\{\begin{array}{cc} 0 & \;\mbox{ for }\;l<1+j\\B_{il}\left(x_{p+1}^1, \ldots, x_{p+1}^{i-l}\right) & \;\mbox{ for }\;l\geq 1+j \end{array}\right.\hspace{.5cm}\mbox{ and }$$
$$\;\;B^p_{il}(x_1, \ldots, x_p)=\left\{\begin{array}{cc} 0 & \;\mbox{ for }\;i<p+l\\B_{lj}^p\left(x_1^1, \ldots, x^{l-j-p+1}_1, \ldots, x_p^1, \ldots, x^{l-j-p+1}_p\right) & \;\mbox{ for }\;i\geq p+l \end{array}\right.,$$
we conclude that $B^{p+1}_{ij}(x_1, \ldots, x_p, x_{p+1})=0$ for $i<(p+1)+j$ and for $i\geq(p+1)+j$,
$$B^{p+1}_{ij}(x_1, \ldots, x_{p+1})=\sum_{p+j\leq l\leq i-1}B_{il}^p\left(x_1^1, \ldots, x^{l-j-p+1}_1, \ldots, x_p^1, \ldots, x^{l-j-p+1}_p\right)B_{lj}\left(x_{p+1}^1, \ldots, x^{l-j-p+1}_{p+1}\right),$$
which only depends on $x_1^1, \ldots, x^{i-j-(p+1)+1}_1, \ldots, x_{p+1}^1, \ldots, x^{i-j-(p+1)+1}_{p+1}$. Therefore, 
$$B^{p+1}_{ij}(x_1, \ldots, x_{p+1})=B^{p+1}_{ij}\left(x_1^1, \ldots, x^{i-j-(p+1)+1}_1, \ldots, x_{p+1}^1, \ldots, x^{i-j-(p+1)+1}_{p+1}\right),$$
concluding the proof.
\end{proof}
\bigskip

By Lemma \ref{derivations}, for any $x_1, \ldots, x_p, y\in\fu$ it holds that 
$$\left(\ad(x_1)\circ\cdots\circ\ad(x_p)y\right)_i=\sum_{j=1}^{i-p}B_{ij}^p\left(x_1^1, \ldots, x_1^{i-j-p+1}, \ldots, x_p^1, \ldots, x_p^{i-j-p+1}\right)y^j.$$
In particular, the $V_i$-component of the bracket
$$[\underbrace{x, [x, \cdots [x}_{r_1}, [\underbrace{y, [y, \cdots [y}_{s_1}, \cdots [\underbrace{x, [x, \cdots [x}_{r_n}, [\underbrace{y, [y, \ldots y}_{s_n}]] \cdots]],$$
depends continuously on the first $i-N$ components of $x$ and $y$, where $N=r_1+s_1+\cdots+r_n+s_n$, proving the following:

\begin{proposition}
\label{product}
For any $x, y\in \fu$, the $V_i$-component of the product $x*y$ satisfies
$$(x*y)^1=x^1+y
^1\hspace{.5cm}\mbox{ and }\hspace{.5cm}(x*y)^i=x^i+y^i+H^i\left(x^1, \ldots, x^{i-1},y^1, \ldots, y^{i-1}\right),$$
where $x=(x^1, \ldots, x^k)$ and $y=(y^1, \ldots, y^k)$ and $H^i:\left(V_1\times\cdots\times V_{i-1}\right)^2\rightarrow V_i$, is a continuous map.
\end{proposition}

Before we prove the main result of this section, let us introduce some notations. For that, let us assume that the derivation $\DC$ has only eigenvalues with nonzero real parts. By the block-matrix decomposition of $\DC$, the same property is true for the linear maps $\DC_{ii}$, that is, $\DC_{ii}$ have only eigenvalues with nonzero real parts. 
Therefore, for any $i\in\{1, \ldots, k\}$ we can decompose $V_i=V_i^+\oplus V_i^-$, where $V_i^{+} (\mbox{resp. }V_i^-)$ is the sum of the real generalized eigenspaces of $\DC_{ii}$ associated with eigenvalues with positive (resp. negative) real parts. As a consequence, for any norm $|\cdot|$ on $\fu$ there exist constants $\kappa_i, \mu_i>0$ such that 
$$|\rme^{t\DC_{ii}}\pi_i^-(x^i)|\leq \kappa_i\rme^{-t\mu_i}|\pi_i^-(x^i)|\;\;\;\;\mbox{ and }\;\;\;\;|\rme^{-t\DC_{ii}}\pi_i^+(x^i)|\leq \kappa_i\rme^{-t\mu_i}|\pi_i^+(x^i)|, \;\;\mbox{ for any }\; t> 0, \;x^i\in V_i,$$
where $\pi_i^{\pm}:V_i\rightarrow V_i^{\pm}$ are the projections associated with the decomposition $V_i=V_i^+\oplus V_i^-$.

\begin{theorem}
	\label{central}
 Under the previous assumptions, let us assume that $\DC$ is hyperbolic and that 
 $H$ is a compact Lie group. Then, there exists $\varepsilon_0, \tau_0>0$ such that $E_{\varepsilon, \tau}$ is a bounded set for all $\varepsilon\in(0, \varepsilon_0)$ and $\tau\geq\tau_0$.
\end{theorem}

\begin{proof} 
Define then $\tau_0>0$ to be a positive real number satisfying, 
$$\forall i\in\{1, \ldots, k\}, \hspace{1cm}\kappa_i\rme^{-\tau_0\mu_i}<1,$$
and $\varepsilon_0>0$ such that 
$$d((e, 0), (h, x))<\varepsilon_0\hspace{.5cm}\implies \hspace{.5cm} |x|<1,$$
where $d$ is a left-invariant metric on $H\times_{\rho}\fu$.

Let us then fix $\varepsilon\in(0, \varepsilon_0)$ and $\tau\geq\tau_0$ and consider the set $E_{\varepsilon, \tau}$. Since $H$ is a compact group, we only have to show that the set $\pi_{i}\left(E_{\varepsilon, \tau}\right)$ is a bounded subset of $V_i$ for any $i=1, \ldots, k$, where 
	$$\pi_{i}:H\times_{\rho}\fu\rightarrow V_i, \;\;\; (h, (x^1, \ldots, x^k))\mapsto x^i,$$
are the projections onto the $V_i$-components. Let $x^i\in \pi_{i}\left(E_{\varepsilon, \tau}\right)$, and consider $(h, x)\in E_{\varepsilon, \tau}$ with $x^i=\pi_i(h, x)$. By definition, there exists a controlled $(\varepsilon, \tau)$-chain from $(e, 0)\in E$ to itself given by $n+m\in\N$, $(h_0, x_0), (h_1, x_1), \ldots, (h_{n+m}, x_{n+m})\in H\times_{\rho}\fu$ with $(h_0, x_0)=(h_{n+m}, x_{n+m})=(e, 0)$, $T_1, \ldots, T_{n+m}\geq\tau$ and $u_0, \ldots, u_{n+m}\in\UC$ such that $(h_n, x_n)=(h, x)$ and
$$d\Bigl((h_{j+1}, x_{j+1}), (h_{T_j, u_j}, x_{T_j, u_j, h_j})\Bigr)<\varepsilon, \hspace{.5cm} i=0, \ldots, n-1,$$
where $s\in\R\mapsto (h_{s, u_j}, x_{s, u_j, h_j})\in H\times_{\rho}\fu,$
stands for the solution of the $\Sigma_{H\times_{\rho}\fu}$ starting at $(h_j, x_j)$. By the left-invariance of the metric, we get that 
$$d\Bigl((h_{j+1}, x_{j+1}), (h_{T_j, u_j}, x_{T_j, u_j, h_j})\Bigr)=d\Bigl((e, 0), (h_{j+1}, x_{j+1})^{-1}(h_{T_j, u_j}, x_{T_j, u_j, h_j})\Bigr)$$
$$=d\Bigl((e, 0), (h_{j+1}^{-1}, -\rho(h_{j+1}^{-1})x_{j+1})(h_{T_j, u_j}, x_{T_j, u_j, h_j})\Bigr)=d\Bigl((e, 0), (h_{j+1}^{-1}h_{T_j, u_j},-\rho(h^{-1}_{j+1})x_{j+1}*\rho(h_{j+1}^{-1})x_{T_j, u_j, h_j})\Bigr)$$
$$=d\Bigl((e, 0), (h_{j+1}^{-1}h_{T_j, u_j},\rho(h^{-1}_{j+1})\left(-x_{j+1}*x_{T_j, u_j, h_j})\right)\Bigr),$$
where for the last equality we used that $\rho(h_{j+1}^{-1})\in\mathrm{Aut}(\fu)$. By the choices made at the beginning of the proof, we obtain that 
$$\left|-x_{j+1}*x_{T_j, u_j, h_j}\right|\leq c\underbrace{\left|\rho(h^{-1}_{j+1})\left(-x_{j+1}*x_{T_j, u_j, h_j}\right)\right|}_{\leq 1}\leq c,$$
where $c:=\max_{h\in H}|\rho(h)|$. Consequently, the same inequality holds for each $V_i$-components, that is, 
$$\Bigl|\Bigl(x_{j+1}*x_{T_j, u_j, h_j}\Bigr)^i\Bigr|<c, \hspace{.5cm}\forall i=1, \ldots, k.$$
By Proposition \ref{product}, we have that 
$$  \Bigl(-x_{j+1}*x_{T_j, u_j, h_j}\Bigr)^i=-x_{j+1}^i+x^i_{T_j, u_j, h_j}
    +H^i\left(-x_{j+1}^1, \ldots, -x_{j+1}^{i-1}, x_{T_j, u_j, h_j}^1, \ldots, x_{T_j, u_j, h_j}^{i-1}\right),$$
    and hence,
\begin{equation}
    \label{eq}
    -x_{j+1}^i+x^i_{T_j, u_j, h_j}=\Bigl(-x_{j+1}*x_{T_j, u_j, h_j}\Bigr)^i
    -H^i\left(-x_{j+1}^1, \ldots, -x_{j+1}^{i-1}, x_{T_j, u_j, h_j}^1, \ldots, x_{T_j, u_j, h_j}^{i-1}\right).
\end{equation}

{\bf Claim:} If there exists a constant $C_i=C_i(\varepsilon, \tau)$ such that 
$$\left|-x^i_{j+1}+x^i_{T_j, u_j, h_j}\right|\leq C_i, \hspace{.5cm}\mbox{ and }\hspace{.5cm}\left|G^i\left(x^1_{s, u_j, h_j}, \ldots, x^{i-1}_{s, u_j, h_j}; Z_{s, u_j, h_j}\right)\right|\leq C_i,$$
for all $j=0, \ldots n+m-1$, and $x^p_{s, u_j, h_j}\in \pi_p(N_{\phi}(E_{\varepsilon, \tau}))$ for all $p=1, \ldots, i-1$, then 
$$|x^i|\leq \frac{2D_i}{(1-\kappa_i\rme^{-\tau\mu_i})}, \hspace{.5cm}D_i:=C_i\left(1+\frac{\kappa_i}{\mu_i}\right).$$

Since,
$$\left|-x^i_{j+1}+x^i_{T_j, u_j, h_j}\right|\leq C_i\hspace{.5cm}\implies\hspace{.5cm}\left|\pi^-_i\left(-x^i_{j+1}+x^i_{T_j, u_j, h_j}\right)\right|\leq C_i,$$
we get that 
$$\left|\pi_i^-(x_{j+1}^i)\right|\leq  \left|\pi_i^-\left(x^i_{T_j, u_j, h_j}\right)\right|+C_i.$$
Using the expression for the $V_i$-component of the solutions of $\Sigma_{H\times_{\rho}\fu}$ given in (\ref{solutiondirect}) we obtain that 
$$\left|\pi_i^-(x_{j+1}^i)\right|\leq C_i+|\rme^{T^i\DC_{ii}}\pi^-_i(x^i_j)|+\int_0^{T_j}\left|\rme^{(T_j-s)\DC_{ii}}\pi_i^-\left(G_i\left(x^1_{s, u_j, h_j}, \ldots, x^{i-1}_{s, u_j, h_j}; Z_{s, u_j, h_j}\right)\right)\right|ds$$
$$\leq C_i+\kappa_i\rme^{-T_j\mu_i}|\pi^-_i(x^i_j)|+\int_{0}^{T_j}\kappa_i\rme^{-(T_j-s)\mu_i}\left|\pi_i^-\left(G^i\left(x^1_{s, u_j, h_j}, \ldots, x^{i-1}_{s, u_j, h_j}; Z_{s, u_j, h_j}\right)\right)\right|ds$$
$$\leq C_i+\kappa_i\rme^{-T_j^i\mu_i}|\pi^-_i(x^i_j)|+\frac{\kappa_i}{\mu_i}C_i(1-\rme^{-T_j^i\mu_i})\leq C_i\left(1+\frac{\kappa_i}{\mu_i}\right)+\kappa_i\rme^{-\tau\mu_i}|\pi^-_i(x^i_j)|,$$
showing that 
$$\left|\pi_i^-(x_{j+1}^i)\right|\leq D_i+\kappa_i\rme^{-\tau\mu_i}|\pi^-_i(x^i_j)|.$$

Since $x_n^i=x^i$ and $x^i_0=0$, the previous estimate applied $n$-times gives us that,
$$\left|\pi^-_i(x^i)\right|\leq D_i+\kappa_i\rme^{-\tau\mu_i}|\pi^-_i(x^i_{n-1})|\leq D_i+\kappa_i\rme^{-\tau\mu_i}\left(D_i+\kappa_i\rme^{-\tau\mu_i}|\pi^-_i(x^i_{n-2})|\right)$$
$$\leq\cdots\leq D_i\sum_{j=0}^{n-1}(\kappa_i\rme^{-\tau\mu_i})^j\leq D_i\sum_{j=0}^{\infty}(\kappa_i\rme^{-\tau\mu_i})^j=\frac{D_i}{(1-\kappa_i\rme^{-\tau\mu_i})}.$$

On the other hand, 
$$-x^i_{j+1}+x^i_{T_j, u_j, h_j}=-x^i_{j+1}+\int_0^{T_j}\rme^{(T_j-s)\DC_{ii}}G^i\left(x^1_{s, u_j, h_j}, \ldots, x^{i-1}_{s, v_j, h_j}; Z_{s, u_j, h_j}\right)ds+\rme^{T_j\DC_{ii}}x^i_j$$
$$=\rme^{T_j\DC_{ii}}\left(-\rme^{-T_j\DC_{ii}}x^i_{j+1}+\int_0^{T_j}\rme^{-s\DC_{ii}}\left(x^1_{s, u_j, h_j}, \ldots, x^{i-1}_{s, u_j, h_j}; Z_{s, u_j, h_j}\right)ds+x^i_j\right),$$
and hence,
$$C_i\geq \left|\pi^+_1\left(-x^i_{j+1}+\int_0^{T_j}\rme^{(T_j-s)\DC_{ii}}G^i\left(x^1_{s, u_j, h_j}, \ldots, x^{i-1}_{s, u_j, h_j}; Z_{s, u_j, h_j}\right)ds+\rme^{T_i\DC_{ii}}x^i_j\right)\right|$$
$$\hspace{2cm}\geq\kappa_1^{-1}\rme^{T_j\mu_i}\left|\pi^+_i\left(-\rme^{-T_j\DC_{ii}}x^i_{j+1}+\int_0^{T_j}\rme^{-s\DC_{ii}}\left(x^1_{s, u_j, h_j}, \ldots, x^{i-1}_{s, u_j, h_j}; Z_{s, u_j, h_j}\right)ds+x^i_j\right)\right|,$$
which gives us that,
$$|\pi^+_i(x_j^i)|\leq C_i\kappa_i\rme^{-T_j\mu_i}+\left|\rme^{-T_j\DC_{ii}}\pi_i^+(x_{j+1}^i)\right|+\int_0^{T_j}\left|\rme^{-s\DC_{ii}}\pi_i^+\left(G^i\left(x^1_{s, u_j, h_j}, \ldots, x^{i-1}_{s, u_j, h_j}; Z_{s, v_j, g_j}\right)\right)\right|ds$$
$$\leq C_i+\kappa_i\rme^{-T_j\mu_i}|\pi^+_i(x^i_{j+1})|+\int_{0}^{T_j}\kappa_i\rme^{-s\mu_i}\left|\pi_i^+\left(G^i\left(x^1_{s, u_j, h_j}, \ldots, x^{i-1}_{s, u_j, h_j}; Z_{s, u_j, h_j}\right)\right)\right|ds$$
$$\leq C_i+\kappa_i\rme^{-T_j\mu_i}|\pi^+_i(x^i_{j+1})|+\frac{\kappa_i}{\mu_i}C_i(1-\rme^{-T_j\mu_i})\leq D_i+\kappa_i\rme^{-\tau\mu_i}|\pi^+_i(x^i_{j+1})|.$$
Since, $x_n^i=x^i$ and $x_{n+m}^i=0$, we obtain that,
$$\left|\pi^+_i(x^i)\right|\leq D_i+\kappa_i\rme^{-\tau\mu_i}|\pi^+_i(x^i_{n+1})|\leq D_i+\kappa_i\rme^{-\tau\mu_i}\left(D_i+\kappa_i\rme^{-\tau\mu_i}|\pi^+_i(x^i_{n+1})|\right)$$
$$\leq\cdots\leq D_i\sum_{j=0}^{m-1}(\kappa_i\rme^{-\tau\mu_i})^j\leq D_i\sum_{j=0}^{\infty}(\kappa_i\rme^{-\tau\mu_i})^j=\frac{D_i}{(1-\kappa_i\rme^{-\tau\kappa_i})}.$$
Consequently,
$$\left|x^i\right|=\left|\pi^+_i(x^i)+\pi^-_i(x^i)\right|\leq \left|\pi^+_i(x^i)\right|+\left|\pi^-_i(x^i)\right|\leq \frac{2D_i}{(1-\kappa_i\rme^{-\tau\kappa_i})},$$
showing the claim.

\bigskip

We can now prove the result by induction: Since $G^1$ is continuous, the compactness of $H\times\UC$ implies the existence of $M_1>0$ such that 
$$|G^1\left(Z_{t, u, h}\right)|\leq M_1, \;\;\;\mbox{ for all }\;\;\;t\geq 0, h\in H, \;u\in\UC.$$
Moreover, in this case 
$$-x^1_{j+1}+x_{T_j, u_j, h_j}^1=\left(-x_{j+1}*x_{T_j, u_j, h_j}\right)^1,$$
which by our initial choice of $\varepsilon, \tau>0$ implies 
$$\left|-x^1_{j+1}+x_{T_j, u_j, h_j}^1\right|=\left|\left(-x_{j+1}*x_{T_j, u_j, h_j}\right)^1\right|\leq c,$$
and hence, $C_1=\max\{M_1, c\}$ implies that $x^1$ is bounded. By the arbitrariness of the initial $x\in E_{\varepsilon, \tau}$ we obtain that $\pi_1(E_{\varepsilon, \tau})$ is bounded.

Let us now assume that $\pi_p(E_{\varepsilon, \tau})$ is bounded for any $p=1, \ldots, i-1$. By Lemma \ref{auxiliar}   we conclude that the sets
$$\left(\prod_{p=1}^{i-1}\pi_p(N_{\phi}(E_{\varepsilon, \tau}))\right)^2\hspace{.5cm}\mbox{ and }\hspace{.5cm}\prod_{p=1}^{i-1}\pi_p(N_{\phi}(E_{\varepsilon, \tau}))\times \ZC,$$ 
are also bounded. Therefore, by continuity, the maps $H^i$ and $G^i$ are also bounded, implying the existence of a constant $C_i>0$ such that 
$$\left|-x^i_{j+1}+x^i_{T_j, u_j, h_j}\right|\leq C_i, \hspace{.5cm}\mbox{ and }\hspace{.5cm}\left|G^i\left(x^1_{s, u_j, h_j}, \ldots, x^{i-1}_{s, u_j, h_j}; Z_{s, u_j, h_j}\right)\right|\leq C_i,$$
where the first inequality comes from the equation (\ref{eq}). By the previous claim, we conclude that  $x^i$ is bounded, and hence the same is true for $\pi_i(E_{\varepsilon, \tau})$, concluding the proof.
\end{proof}

As a direct corollary, we have the following:

\begin{corollary}
    Under the assumptions of the previous theorem, the chain control set of the LCS $\Sigma_{H\times_{\rho}\fu}$ is a compact subset.
\end{corollary}

\subsection{Compactness of the chain control set}

This section concludes our results by showing that the compactness of the central subgroup associated with a LCS on a connected Lie group is a necessary and sufficient condition for the compactness of its unique chain control set. Precisely, we have the following:

\begin{theorem}
Let $\Sigma_G$ be an LCS on a connected group $G$. If the central subgroup associated with $\Sigma_G$ is compact, then $\Sigma_G$ admits a unique and compact chain control set.

\end{theorem}

\begin{proof}
    By Theorem \ref{unique}, if $G^0$ is a compact subgroup, there exists a unique control set $E$ that contains $G^0$. Using Proposition \ref{conj}, there exists a homomorphism $\psi:G\rightarrow H\times_{\rho}\fu$, where $H$ is a compact subgroup and $\fu$ a simply connected, connected nilpotent Lie group. Moroever, the linear vector field $\XC$ of $\Sigma_G$ is $\psi$ conjugated to a linear vector field on $H\times_{\rho}\fu$ of the form $\widehat{\XC}\times\widehat{\DC}$, where $\widehat{\DC}$ has only eigenvalues with nonzero real parts. 

    By considering the LCS on $\Sigma_{H\times_{\rho}\fu}$ induced by $\psi$, Theorem \ref{central} implies that the chain control set $F$ of $\Sigma_{H\times_{\rho}\fu}$ is compact, which by Proposition \ref{control} implies that $E$ is a compact subset, concluding the proof.
\end{proof}

\end{document}